\documentclass[12pt]{article}
\usepackage{amsmath,amsfonts}
\usepackage{xcolor}
\usepackage{tikz}
\usepackage{graphicx}
\numberwithin{equation}{section}
\newtheorem{theorem}{Theorem}[section]
\newtheorem{lemma}[theorem]{Lemma}

\newcommand{\qed}{\nolinebreak\hfill\vbox{\hrule\hbox{\vrule\kern3pt\vbox{\kern6pt}\kern3pt\vrule}\hrule}}
\newenvironment{pf}{{\noindent\bf Proof.}}{\qed\newline}

\begin{document}

\newcommand{\avint}{{- \hspace{-3.5mm} \int}}

\title{The Bifurcation Phenomenon for the Regularized Two-Phase Problem Associated with the $p$-Laplacian}
\author{Alaa Haj Ali,\thanks{Department of Mathematics, University of Michigan, Ann Arbor, MI 48109, USA}
~ 
Catherine Lebiedzik\thanks{Department of Mathematics, Wayne State University, 656 W Kirby, Detroit, MI 48202, USA. Corresponding author, email: ar6554@wayne.edu }
~~\&
Peiyong Wang\thanks{Department of Mathematics, Wayne State University, 656 W Kirby, Detroit, MI 48202, USA.  Peiyong Wang is partially supported by a Simon's Collaboration Grant. }
}

\date{}
\maketitle
\begin{abstract}
In this paper, we verify a bifurcation phenomenon regarding the multiplicity of weak solutions, subject to the Dirichlet boundary condition, of a regularized two-phase free boundary problem associated with the $p$-Laplacian. In fact, we prove the existence of a mountain pass solution when the boundary data is small. In this case, three weak solutions of the regularized problem exist: the $p$-harmonic function, a minimizer of the corresponding functional, and the mountain pass solution. Finally, we consider the special case of radially symmetric solutions.  By solving the associated nonlinear ordinary differential equation in the unregularized case, we show explicitly how the Bernoulli condition at the free boundary gives rise to the bifurcation of solutions. 
\end{abstract}

\textbf{AMS Classifications:} 35R35, 35B32, 35A02, 35J62, 35J25
\medskip

\textbf{Keywords:} Bifurcation, $p$-Laplacian, free boundary problem, Mountain Pass Theorem, Fr\'{e}chet derivative

\newpage


\section{Introduction}\label{introduction}
In this paper, we study a bifurcation in a Dirichlet boundary value problem derived from the functional $J_{p,\varepsilon}$ defined by
\begin{equation}\label{functional}
J_{p,\varepsilon}(u) = \int_{\Omega} \frac{1}{p}|\nabla u|^p\,\Gamma_{\varepsilon}(u)
+ \left(\frac{1}{p}|\nabla u|^p + u\right)\Theta_{\varepsilon}(u) + q(x)\lambda_{\varepsilon}(u)\,dx
\end{equation}
which is a regularized version of the functional
\begin{equation}\label{fbp}
\begin{split}
J_p(u) = &\int_{\Omega}\left(\frac{1}{p}|\nabla u(x)|^p + q(x)\lambda^p_2\right)\chi_{\{u>0\}}(x) \\
&+ \left(\frac{1}{p}|\nabla u(x)|^p + u(x) + q(x)\lambda^p_1 \right)\chi_{\{u\leq 0\}}\,dx.
\end{split}
\end{equation}
Here, $p > 1$, $\Omega\subset\mathbb{R}^n$ is a smooth bounded domain, $q(x) > 0$ is a weight function, $0 < \lambda_1 < \lambda_2$ are constants, and $\Gamma_{\varepsilon}$ and $\Theta_{\varepsilon}\in C^{\infty}_0(\mathbb{R}\rightarrow [0,1])$, and $\lambda_{\varepsilon}\in C^{\infty}_0(\mathbb{R}\rightarrow [\lambda_1, \lambda_2])$ are smooth functions that satisfy
\begin{equation}\label{fcts1}
\Gamma_{\varepsilon}(s) + \Theta_{\varepsilon}(s) = 1
\end{equation}
\begin{equation}\label{fcts2}
\Gamma_{\varepsilon}(s) = \left\{\begin{array}{ll} 0 &\ \ \text{if\ } s\leq 0\\ 1 &\ \ \text{if\ } s\geq \varepsilon \end{array}\right.
\end{equation}
\begin{equation}\label{fcts3}
\lambda_{\varepsilon}(s) = \left\{\begin{array}{ll} \lambda^p_1 &\ \ \text{if\ } s\leq 0\\ \lambda^p_2 &\ \ \text{if\ } s\geq \varepsilon. \end{array}\right.
\end{equation}
We can readily rewrite the functional $J_{p,\varepsilon}$ in the form
\begin{equation}\label{functional-1}
J_{p,\varepsilon}(u) = \int_{\Omega} \frac{1}{p}|\nabla u|^p + u\Theta_{\varepsilon}(u) + q(x)\lambda_{\varepsilon}(u)\,dx.
\end{equation}

The Euler-Lagrange equation of the regularized problem (\ref{functional-1}) is expressed as
\begin{equation}\label{euler}
-\Delta_p u + \theta_{\varepsilon}(u)u + \Theta_{\varepsilon}(u) + q(x)\mu_{\varepsilon}(u) = 0\ \ \text{in\ } \Omega
\end{equation}
where $\theta_{\varepsilon}(s) = \Theta'_{\varepsilon}(s)$, and $\mu_{\varepsilon}(s) = \lambda'_{\varepsilon}(s)$. The \textbf{$p$-Laplacian} for a $C^2$ function $u$ is defined by
\begin{equation*}
\Delta_p u = \mathrm{div}\left(|\nabla u|^{p-2}\nabla u\right) = (p-2)|\nabla u|^{p-4}\left<D^2u\nabla u, \nabla u\right> + |\nabla u|^{p-2}\Delta u
\end{equation*}
in the divergence and non-divergence forms, respectively. We complete the equation (\ref{euler}) into a Dirichlet boundary value problem by imposing
\begin{equation}\label{bdry}
u = \sigma\ \ \text{on\ } \partial\Omega
\end{equation}
for a given positive function $\sigma$ in the Sobolev space $W^{1,p}(\Omega)$ such that $0 < \varepsilon < \inf_{\partial\Omega} \sigma(x)$ and $J_{p,\varepsilon}(\sigma) < \infty$. 

It is well-known that a minimizer of the functional $J_p$ satisfies in the weak sense the two-phase free boundary problem associated with the $p$-Laplacian:
\begin{equation}\label{unreg}
\left\{\begin{array}{ll}\Delta_pu = 0 &\ \ \text{in\ }\{u>0\},\\ \Delta_pu = 1 &\ \ \text{in\ }\{u\leq 0\}^{\circ},\\ \left(u^+_{\nu}\right)^p = \left(u^-_{\nu}\right)^p + Q(x)\left(\lambda^p_2 - \lambda^p_1\right) &\ \ \text{on\ }\partial\{u>0\}\cap\Omega,\\ u = \sigma &\ \ \text{on\ }\partial\Omega,\end{array}\right.
\end{equation}
for some positive function $Q$. Equations of (\ref{unreg}) represent a two-phase Bernoulli type free boundary problem. Therefore, our boundary value problem (\ref{euler})--(\ref{bdry}) is a regularized two-phase free boundary problem associated with the $p$-Laplacian with a fattened free boundary. Two-phase Bernoulli type free boundary problems have been under extensive research since the 1980's. Since the publication of seminal papers \cite{ACF} and \cite{C1}--\cite{C3}, the study of the regularity and existence of a solution of the two-phase problem has continued to draw attention. Vast efforts have since been made to extend the theory to non-Laplacian linear, nonlinear, and degenerate situations. The survey papers \cite{DFS} and \cite{FLS} provide details of the developments in this regard.

In this paper, we are concerned with the number of weak solutions to the free boundary problem. In fact, we establish a bifurcation regarding the existence and multiplicity of solutions of the two-phase problem (\ref{euler})--(\ref{bdry}) for the degenerate $p$-Laplacian. Bifurcation phenomena of this type for the regularized one-phase problem for the Laplace equation or the more general $p$-Laplace equation were proved in \cite{CW} and \cite{HW}. Later, the bifurcation for the regularized two-phase problem for the Laplace equation was also justified in \cite{CHRTW}. Here, we need to overcome the technical gaps that differentiate the $p$-Laplacian from the Laplacian to show the existence of the bifurcation. 

The main result of this paper is Theorem \ref{thirdsolution}, which states that problem (\ref{euler})-(\ref{bdry}) has three weak solutions, namely a trivial solution, a functional minimizer and a mountain-pass solution for relatively small boundary data. This type of existence result has been proved for the regularized versions of the one-phase and two-phase free boundary problems associated with the Laplacian (Theorem 2.2 of \cite{CW} and Theorem 2.2 of \cite{HW}), and the one-phase problem for the $p$-Laplacian (Theorem 2.3 of \cite{CHRTW}).

The study of the bifurcation originates in our efforts to interpret the physical meaning of the mathematical model. Mathematically, an evolutionary Bernoulli-type problem carries the uniqueness of a solution, while the limiting stationary problem generally does not, \cite{LW}. So, it is important to know when and how there are multiple solutions to the stationary problem. In addition, when there are multiple solutions to the stationary problem, it is meaningful to identify each solution according to its respective stability. These are the goals of the current, previous (\cite{CW}, \cite{HW}, \cite{CHRTW}) and subsequent papers.

Finally, as an illustration of this bifurcation phenomenon, in this paper we consider the special
case of radially symmetric solutions to the unregularized problem \eqref{unreg} with
$\Omega = B_1(0)$ and $Q(x) \equiv 1$. In this setting the boundary value problem
\eqref{euler}--\eqref{bdry} reduces to a nonlinear ordinary differential equation on
each side of the free radius $r_0$, and we explicitly show that the Bernoulli condition
at the free boundary can be rearranged into a relationship $\sigma = f(r_0)$ between
the prescribed boundary data $\sigma$ and the geometric datum $r_0$ locating the free
boundary. Our computation shows that for a given value of the boundary data $\sigma$,
the equation $\sigma = f(r_0)$ admits zero, one, or two solutions $r_0 \in (0,1)$;
according to whether $\sigma$ lies above, at, or below the maximum value of $f$. While
this unregularized, radially symmetric case falls outside the scope of Theorem~2.2, the
explicit multiplicity established here motivates our expectation that an analogous
bifurcation result holds for the unregularized problem on general domains.

\section{The Main Theorem and Its Proof}
We denote the trivial solution of the problem (\ref{euler})--(\ref{bdry}) by $u_0$. It is given by the $p$-harmonic function with boundary data $\sigma$, i.e., the weak solution of boundary value problem
\begin{equation*}
\left\{\begin{array}{ll}
-\Delta_p u = 0 &\ \text{\ in\ }\Omega\\
u = \sigma &\ \text{\ on\ }\partial\Omega,\end{array}\right.
\end{equation*}

We now show a minimizer $u_2$ of the functional $J_{p,\varepsilon}$ exists, which is also a weak solution of (\ref{euler})--(\ref{bdry}). We include a short proof for the sake of completeness, although the argument is standard. Notice that, in general, a minimizer of $J_{p,\varepsilon}$ is not unique. 
\begin{theorem}
There exists a minimizer of the functional $J_{p,\varepsilon}$.
\end{theorem}
\begin{pf}
Define
\begin{equation*}
\mathcal{A} = \left\{v\in W^{1,p}(\Omega)\colon v - \sigma \in W^{1,p}_0(\Omega)\right\}
\end{equation*}
and
\begin{equation*}
m = \inf_{v\in\mathcal{A}} J_{p,\varepsilon}(v).
\end{equation*}
Clearly, the set $\{J_{p,\varepsilon}(v)\colon v\in\mathcal{A}\}$ is bounded below. Let $\{u_k\}$ be a minimizing sequence of $J_{p,\varepsilon}$. Then it is a bounded sequence in $W^{1,p}(\Omega)$ as $\Omega$ has finite measure. Alaoglu's Theorem implies that, for a subsequence still denoted by $\{u_k\}$ for convenience, there is a certain function $u\in W^{1,p}(\Omega)$ with $u-\sigma\in W^{1,p}_0(\Omega)$ that satisfies
\begin{enumerate}
\item $\nabla u_k \rightharpoonup \nabla u$ in $L^p(\Omega)$;
\item $u_k\rightarrow u$ a.\,e.\,in $\Omega$; and
\item $u_k\Theta(u_k) + q(x)\lambda_{\epsilon}(u_k)\rightarrow u\Theta(u) + q(x)\lambda_{\epsilon}(u)$ in the weak-$*$ sense in $L^{\infty}_{\mathrm{loc}}(\Omega)$.
\end{enumerate}
Consequently, Fatou's lemma leads to
\begin{equation}
J_{p,\varepsilon}(u)\leq \liminf_{k\rightarrow\infty} J_{p,\varepsilon}(u_k) = m,
\end{equation}
and we are done.
\end{pf}

Besides the two solutions $u_0$ and $u_2$ for relatively small boundary data, the problem (\ref{euler})-(\ref{bdry}) has a  third weak solution of the mountain-pass type, which will be denoted by $u_1$. In essence, the mountain-pass lemma is a way to produce a saddle point solution. The main theorem of this paper is the following one.

\begin{theorem}\label{thirdsolution}
If $\varepsilon << \sigma_m$ and $J_{p,\varepsilon}[u_2] < J_{p,\varepsilon}[u_0]$, then there is
a third weak solution $u_1$ of the boundary value problem (\ref{euler})-(\ref{bdry}). Moreover, $J_{p,\varepsilon}[u_1]
\geq J_{p,\varepsilon}[u_0] + a$ for some $a > 0$ which is independent of $\varepsilon$.
\end{theorem}

Let $\sigma_M = \max_{\partial\Omega}\sigma(x)$ and $\sigma_m = \min_{\partial\Omega}\sigma(x)$. We first prove that if $\sigma_M$ is sufficiently small, then $J_{p,\varepsilon}[u_2] < J_{p, \varepsilon}[u_0]$ and hence $u_0\neq u_2$. In fact, we may pick $u\in W^{1,p}(\Omega)$ so that
\begin{equation*}
\left\{\begin{array}{ll}
u = 0 &\ \text{\ in\ }\overline{\Omega}_{\delta}\\
u = \sigma &\ \text{\ on\ }\partial\Omega,\ \text{and}\\
-\Delta_p u = 0 &\ \text{\ in\ }\Omega\backslash\overline{\Omega}_{\delta},
\end{array}\right.
\end{equation*}
where $\Omega_{\delta} = \{x\in\Omega: dist(x, \partial\Omega) > \delta\}$ and $\delta > 0$
is a small constant independent of $\varepsilon$ and $\sigma$ so that the integral $\int_{\Omega_{\delta}}
q(x)dx > 0$ and its value is also independent of $\varepsilon$ and $\sigma$. Without the loss of generality, we may assume $\Omega_{\delta}$ is smooth,
since otherwise we may approximate it with a smooth domain in the argument. Then,
\begin{equation*}
J_{p,\varepsilon}(u_0) \geq \int_{\Omega}\frac{1}{p}|\nabla u_0|^p + q(x)\lambda^p_2\,dx
\geq \int_{\Omega}q(x)\lambda^p_2\,dx.
\end{equation*}
As $|\nabla u| \leq C\frac{\sigma_M}{\delta}$ in $\Omega\backslash\overline{\Omega}_{\delta}$ from the gradient estimates (e.\,g.\,Chapter 4, \cite{LU} or Chapter 14, \cite{GT}) and $u\Theta_{\varepsilon}(u)\leq\varepsilon$ for $u\geq 0$, we know that
\begin{alignat*}{1}
&J_{p,\varepsilon}(u) = \int_{\Omega}\frac{1}{p}|\nabla u|^p + u\Theta_{\varepsilon}(u) + q(x)\lambda_{\varepsilon}(u)\,dx\\
&\leq \int_{\Omega}\frac{1}{p}|\nabla u|^p + \varepsilon|\Omega| + \int_{\Omega\backslash\Omega_{\delta}}q(x)\lambda^p_2 + \int_{\Omega_{\delta}}q(x)\lambda^p_1\\
&\leq C\int_{\Omega\backslash \Omega_{\delta}}\frac{\sigma^p_M}{\delta^p} + \varepsilon|\Omega| + \int_{\Omega\backslash\Omega_{\delta}}q(x)\lambda^p_2 + \int_{\Omega_{\delta}}q(x)\lambda^p_1.
\end{alignat*}
So, for all small $\varepsilon > 0$,
\begin{equation*}
J_{p,\varepsilon}(u) - J_{p,\varepsilon}(u_0)
\leq C\int_{\Omega\backslash\Omega_{\delta}}\frac{\sigma^p_M}{\delta^p} + \varepsilon|\Omega| - \int_{\Omega_{\delta}}q(x)\left(\lambda^p_2 - \lambda^p_1\right)\,dx < 0
\end{equation*}
if $\sigma_M\leq \sigma_0$ for some $\sigma_0 = \sigma_0(\delta, \Omega, q)$ small enough. In particular, $J_{p,\varepsilon}(u_2) \leq J_{p,\varepsilon}(u) < J_{p,\varepsilon}(u_0)$, and hence $u_2\neq u_0$.

Let $\mathfrak{B}$ denote the Banach space $W^{1,p}_0(\Omega)$ that we will work with. For every $v\in\mathfrak{B}$, we write $u = v + u_0$ and adopt the norm $\|v\|_{\mathfrak{B}} = \left(\int_{\Omega}|\nabla v|^p\right)^{\frac{1}{p}} = \left(\int_{\Omega}|\nabla u - \nabla u_0|^p\right)^{\frac{1}{p}}$. We define the functional
\begin{equation}
I[v] = J_{p,\varepsilon}(u) - J_{p,\varepsilon}(u_0) = \int_{\Omega}\frac{1}{p}|\nabla u|^p + u\Theta_{\varepsilon}(u) - \int_{\{u < \varepsilon\}}q(x)\left(\lambda^p_2 - \lambda_{\varepsilon}(u)\right) - \int_{\Omega}\frac{1}{p} |\nabla u_0|^p
\end{equation}
Set $v_2 = u_2 - u_0$. Clearly, $I[0] = 0$ and $I[v_2] \leq 0$ on account of the definition of $u_2$ as a minimizer of $J_{p,\varepsilon}$. If $I[v_2] < 0$ which is the case when $\sigma_M$ is small, we will apply the Mountain Pass Lemma to prove that there exists a critical point of the functional $I$ which is a weak solution of the problem (\ref{euler})-(\ref{bdry}).

The Fr\'{e}chet derivative of $I$ at $v = u - u_0\in \mathfrak{B}$ is given by
\begin{equation}
I'[v]\varphi = \int_{\Omega}|\nabla u|^{p-2}\nabla u\cdot\nabla\varphi + \left(u\theta_{\varepsilon}(u) + \Theta_{\varepsilon}(u) + q(x)\mu_{\varepsilon}(u)\right)\varphi,\ \ \ \varphi\in\mathfrak{B}
\end{equation}
which is obviously in the dual space $\mathfrak{B}^*$ of $\mathfrak{B}$ in light of the H\"{o}lder's inequality. Equivalently
\begin{equation}
I'[v] = -\Delta_p (v+u_0) + (v+u_0)\theta_{\varepsilon}(v+u_0) + \Theta_{\varepsilon}(v+u_0) + q(x)\mu_{\varepsilon}(v+u_0)\in\mathfrak{B}^*.
\end{equation}

To prove $I'$ is Lipschitz continuous on any bounded subset of $\mathfrak{B}$ with Lipschitz constant depending on $\varepsilon$, $p$, $|\Omega|$, and $\sup q$, we follow the argument in \cite{HW}. In fact, for any $v$, $w$, and $\varphi\in \mathfrak{B}$,
\begin{equation}
\left|I'[v]\varphi - I'[w]\varphi\right| \leq D_1 + D_2,
\end{equation}
where
\begin{equation}
D_1 = \left|\int_{\Omega}\left|\nabla v + \nabla u_0\right|^{p-2}(\nabla v + \nabla u_0)\cdot\nabla\varphi - \left|\nabla w + \nabla u_0\right|^{p-2}(\nabla w + \nabla u_0)\cdot\nabla\varphi\right|
\end{equation}
and
\begin{equation}
\begin{split}
D_2 = &\left|\int_{\Omega}\left[(v+u_0)\theta_{\varepsilon}(v+u_0)+\Theta_{\varepsilon}(v+u_0)+q(x)\mu_{\varepsilon}(v+u_0)\right]\varphi\right.\\
&\left. - \left[(w+u_0)\theta_{\varepsilon}(w+u_0)+\Theta_{\varepsilon}(w+u_0)+q(x)\mu_{\varepsilon}(w+u_0)\right]\varphi\right|.
\end{split}
\end{equation}
$D_1$ is estimated as in \cite{HW}
\begin{equation}
D_1 \leq C(p)\left(\|\nabla v\|_{L^p} + \|\nabla w\|_{L^p} + \|\nabla u_0\|_{L^p}\right)^{p-2}\|\nabla v - \nabla w\|_{L^p(\Omega)}\|\nabla\varphi\|_{L^p(\Omega)}.
\end{equation}
On the other hand, it is not difficult to get from the Young's inequality that
\begin{equation}
D_2 \leq C(1/\varepsilon,p,|\Omega|,\sup q)\|v-w\|_{L^p(\Omega)}\|\varphi\|_{L^p(\Omega)}.
\end{equation}
These estimates readily imply the Lipschitz continuity of $I'$ on bounded subsets of $\mathfrak{B}$.

Next we justify the Palais-Smale condition on the functional $I$. Suppose $\{v_k\}\subset\mathfrak{B}$ is a Palais-Smale sequence in the sense that
\begin{equation*}
\left|I[v_k]\right|\leq M\ \ \ \ \text{and\ \ }\ \ I'[v_k]\rightarrow 0\ \ \ \ \text{in $\mathfrak{B}^*$}
\end{equation*}
for some $M > 0$. Let $u_k = v_k + u_0\in W^{1,p}(\Omega)$, $k = 1, 2, 3, ...$.

We firstly note that for $v\in W^{1,p}_0(\Omega)$ and $u = v + u_0$, the functions $\theta_{\varepsilon}(u)$, $\Theta_{\varepsilon}(u)$, and $\mu_{\varepsilon}(u)$ are in $W^{1,p}_0(\Omega)$ by their definition. Then, since $u\theta_{\varepsilon}(u) + \Theta_{\varepsilon}(u) + q(x)\mu_{\varepsilon}(u)\in W^{1,p}_0(\Omega)$ with $u = v + u_0$ and $v\in W^{1,p}(\Omega)$, the Rellich-Kondrachov Compactness Theorem that states $W^{1,p}_0(\Omega)\subset\subset L^p(\Omega)\subset \mathfrak{B}^*$ implies the mapping $v\mapsto u\theta_{\varepsilon}(u) + \Theta_{\varepsilon}(u) + q(x)\mu_{\varepsilon}(u)$ from $W^{1,p}_0(\Omega)$ to $\mathfrak{B}^*$ is compact. Consequently, there exists an $f\in L^p(\Omega)\subset\mathfrak{B}^*$ such that, restricted to a subsequence of $\{u_k := v_k + u_0\}$ if necessary,
\begin{equation*}
u_k\theta_{\varepsilon}(u_k) + \Theta_{\varepsilon}(u_k) + q(x)\mu_{\varepsilon}(u_k)\rightarrow -f\ \ \text{ in $L^p(\Omega)$.}
\end{equation*}
Recall that
\begin{equation*}
\begin{split}
&\ \ \ \ \left\|I'[v_k]\right\| \\
&= \sup_{\|\varphi\|_{\mathfrak{B}}\leq 1}\left|\int_{\Omega}|\nabla u_k|^{p-2}\nabla u_k\cdot\nabla\varphi + \left(u_k\theta_{\varepsilon}(u_k) + \Theta_{\varepsilon}(u_k) + q(x)\mu_{\varepsilon}(u_k)\right) \varphi\right|\rightarrow 0.
\end{split}
\end{equation*}
As a consequence,
\begin{equation}\label{test1}
\sup_{\|\varphi\|_{\mathfrak{B}}\leq M}\left|\int_{\Omega}|\nabla u_k|^{p-2}\nabla u_k\cdot\nabla \varphi - f\varphi\right| \rightarrow 0\ \ \ \ \text{for any $M\geq 0$.}
\end{equation}
Obviously, that $\{I[v_k]\}$ is bounded implies that a subsequence of $\{v_k\}$, still denoted by $\{v_k\}$ by abusing the notation without confusion, converges weakly in $\mathfrak{B} = W^{1, p}_0(\Omega)$. In particular,
\begin{equation*}
\int_{\Omega}fv_k - fv_m\rightarrow 0\ \ \ \ \text{as $k$, $m\rightarrow\infty$.}
\end{equation*}
Then by setting $\varphi = v_k - v_m = u_k - u_m$ in (\ref{test1}), one gets
\begin{equation}\label{conv}
\left|\int_{\Omega}\left(|\nabla u_k|^{p-2}\nabla u_k - |\nabla u_m|^{p-2}\nabla u_m\right)\cdot \nabla (u_k - u_m)\right| \rightarrow 0\ \ \ \ \text{as $k$, $m\rightarrow\infty$,}
\end{equation}
since
\begin{equation*}
\|u_k - u_m\|^p_{\mathfrak{B}} = \|v_k - v_m\|^p_{\mathfrak{B}} \leq 2pM + 2J_p[u_0].
\end{equation*}
In particular, if $p = 2$, $\{v_k\}$ is a Cauchy sequence in $W^{1,2}_0(\Omega)$ and hence converges.
We will apply the following elementary inequalities associated with the $p$-Laplacian, \cite{L}, to the general case $p\neq 2$:
\begin{alignat}{1}
&\langle |b|^{p-2}b - |a|^{p-2}a,\,b - a\rangle \geq (p-1)|b-a|^2(1 + |a|^2 + |b|^2)^{\frac{p-2}{2}},\ \ 1\leq p\leq 2;\label{ele1}\\
&\text{and}\ \ \ \ \langle |b|^{p-2}b - |a|^{p-2}a,\,b - a\rangle \geq 2^{2-p}|b-a|^p,\ \ p\geq 2.\label{ele2}
\end{alignat}
We assume first $1 < p < 2$. Let $K = 2pM + 2J_p[u_0]$. Then the first elementary inequality (\ref{ele1}) implies
\begin{alignat*}{1}
&\ \ \ \ \ (p-1)\int_{\Omega}|\nabla u_k - \nabla u_m|^2\left(1+|\nabla u_k|^2 + |\nabla u_m|^2\right)^{\frac{p-2}{2}}\\ &\leq \int_{\Omega}\left(|\nabla u_k|^{p-2}\nabla u_k - |\nabla u_m|^{p-2}\nabla u_m\right)\cdot \nabla (u_k - u_m) \rightarrow 0
\end{alignat*}
Meanwhile H\"{o}lder's inequality implies
\begin{alignat*}{1}
&\ \ \ \ \ \int_{\Omega}|\nabla v_k - \nabla v_m|^p = \int_{\Omega}|\nabla u_k - \nabla u_m|^p \\
&\leq \left(\int_{\Omega}|\nabla u_k - \nabla u_m|^2\left(1 + |\nabla u_k|^2 + |\nabla u_m|^2\right)^{\frac{p-2}{2}}\right)^{\frac{p}{2}}
\left(\int_{\Omega}\left(1 + |\nabla u_k|^2 + |\nabla u_m|^2\right)^{\frac{p}{2}}\right)^{\frac{2-p}{2}} \\
&\leq C(p)\left(|\Omega| + K\right)^{\frac{2-p}{2}} \left(\int_{\Omega}|\nabla u_k - \nabla u_m|^2\left(1 + |\nabla u_k|^2 + |\nabla u_m|^2\right)^{\frac{p-2}{2}}\right)^{\frac{p}{2}}
\end{alignat*}
Therefore, $\{v_k\}$ is a Cauchy sequence in $\mathfrak{B}$ and hence converges.

Suppose $p > 2$. The second elementary inequality (\ref{ele2}) implies
\begin{alignat*}{1}
&\ \ \ \ \ \int_{\Omega}|\nabla v_k - \nabla v_m|^p = \int_{\Omega}|\nabla u_k - \nabla u_m|^p \\
&\leq 2^{p-2}\int_{\Omega}\left(|\nabla u_k|^{p-2}\nabla u_k - |\nabla u_m|^{ p-2}\nabla u_m\right)\cdot \left(\nabla u_k - \nabla u_m\right),
\end{alignat*}
which in turn implies $\{v_k\}$ is a Cauchy sequence in $\mathfrak{B}$ and hence converges, on account of (\ref{conv}). The Palais-Smale condition is verified for $1 < p < \infty$ for the functional $I$ on the Banach space $W^{1,p}_0(\Omega)$.

We will apply in the main proof the following lemma that follows from the Fundamental Theorem of Calculus.
\begin{lemma}\label{p-inequalities}
For any $a$ and $b\in\mathbb{R}^n$, it holds
\begin{equation}\label{ele3}
|b|^p \geq |a|^p + p\langle |a|^{p-2}a, b-a\rangle +\, C(p)|b - a|^p\ \ \ \ (p\geq 2)
\end{equation}
where $C(p) > 0$.

If $1 < p < 2$, then
\begin{equation}\label{ele4}
|b|^p \geq |a|^p + p\langle |a|^{p-2}a, b-a\rangle  +\, C(p)|b-a|^2\int^1_0\int^t_0\left|(1-s)a+sb\right|^{p-2}\,dsdt,
\end{equation}
where $C(p) = p(p-1)$.
\end{lemma}

Back to the main proof, we are now in a position to show there is a closed mountain ridge around the origin of the Banach space $\mathfrak{B}$ that separates $v_2$ from the origin with the energy $I$ as the elevation function. This is the content of the following lemma.
\begin{lemma}
For all small $\varepsilon > 0$ such that $C\varepsilon \leq \frac{1}{2}\sigma_m$ for a large universal constant $C$, there exist positive constants $\delta$ and $a$ independent of $\varepsilon$, such that, for every $v$ in $\mathfrak{B}$ with $\|v\|_{\mathfrak{B}} = \delta$, the inequality $I[v] \geq a$ holds.
\end{lemma}
\begin{pf}
It suffices to prove $I[v] \geq a > 0$ for every $v\in C^{\infty}_0(\Omega)$ with $\|v\|_{\mathfrak{B}} = \delta$ for $\delta$ small enough, as $I$ is continuous on $\mathfrak{B}$, and $C^{\infty}_0(\Omega)$ is dense in $\mathfrak{B}$.

For such a function $v$, define $u = v + u_0$ and the set $\Lambda = \Lambda_{\varepsilon} = \{u\leq\varepsilon\}$ which will be shown to be empty for sufficiently small $\delta$, albeit $\delta \gg \varepsilon$. Let $\mathcal{AC}([0,1],S)$, where $S\subseteq\mathbb{R}^n$, be the set of absolutely continuous functions $\gamma: [0,1]\rightarrow S$. For each $\gamma\in\mathcal{AC}([0,1],S)$, its length is defined as $L(\gamma) = \int^1_0| \gamma'(t)|\,dt$. For each $x_0\in\partial\Omega$, we define the distance from $x_0$ to $\Lambda$ to be
\begin{equation}
d(x_0,\Lambda) = \inf\{L(\gamma):\ \gamma\in\mathcal{AC}([0,1],\overline{\Omega}),\ \text{s.\,t.\,}\gamma(0) = x_0,\ \text{and\ }\gamma(1)\in\Lambda\}
\end{equation}
It was proved in \cite{CW} and \cite{CHRTW} that there is a minimizing path $\gamma$ in $\mathcal{AC}([0,1],\overline{\Omega})$ that realizes the distance $d(x_0,\Lambda)$.

Assume $\Lambda \neq \emptyset$. We replace $\Lambda$ by a small ball on which $u(x) \leq C\varepsilon \ll\sigma_m$. We still denote this set by $\Lambda$ and let $z$ be its center. Further, we define the \textbf{directly accessible boundary} as the set
\begin{equation*}
\mathcal{DA}(\partial\Omega) = \left\{x_0\in\partial\Omega: \exists \text{\ a minimizing $\gamma$ of $d(x_0,\Lambda)$ satisfying $\gamma(t)\in\Omega\backslash\Lambda$ for $t\in(0,1)$.}\right\}
\end{equation*}
We note that for any $x_0\in\mathcal{DA}(\partial\Omega)$ a minimizing path of $d(x_0,\Lambda)$ is a straight line segment. In fact, the minimizing path would be a polygonal curve with endpoints of each segment belonging to $\partial\Omega\cup\Lambda$ otherwise, and hence a contradiction occurs if the path consists of more than one segment.

Let $S$ be the union of the minimizing paths for all the points in the directly accessible boundary $\mathcal{DA}(\partial\Omega)$ and $\Omega_1$ be the interior of $S$. It is evident that $|\Omega_1| = H^n(\Omega_1) > 0$ and $H^{n-1}(\mathcal{DA}(\partial\Omega)) > 0$, since there is one line that passes through any given point on $\partial\Lambda$ from the center of $\Lambda$ to an accessible boundary point on $\mathcal{DA}(\partial\Omega)$. Furthermore, these straight paths reach every point on $\mathcal{DA}(\partial\Omega)$. Clearly, any two such line segment paths do not intersect in $\Omega\backslash\Lambda$. Let $\gamma = \gamma_{x_0}$ be such a path with $\gamma(0) = x_0\in\mathcal{DA}(\partial\Omega)$ and $\gamma(1)\in\Lambda$. Then $v(x_0) = 0$ and $v(\gamma(1)) \leq C\varepsilon - u_0(\gamma(1)) \leq C\varepsilon - \sigma_m < 0$. So, the Fundamental Theorem of Calculus $$v(\gamma(1)) - v(\gamma(0)) = \int^1_0\nabla v(\gamma(t))\cdot\gamma'(t)\,dt$$ implies that
\begin{equation}\label{2.17}
\sigma_m - C\varepsilon \leq \int^1_0|\nabla v(\gamma(t))|\,|\gamma'(t)|\,dt.
\end{equation}
For each $x_0\in\mathcal{DA}(\partial\Omega)$, let $e(x_0)$ be the unit vector in the direction of $x_0 - z$ and $\nu(x_0)$ the outer normal to $\partial\Omega$ at $x_0$. Then $\nu(x_0)\cdot e(x_0) > 0$ everywhere on $\mathcal{DA}(\partial\Omega)$, since the segment $\overline{x_0z}\subset\Omega$. Hence the above inequality (\ref{2.17}) implies
\begin{equation*}
\begin{split}
&\ \ \ \ \ \ (\sigma_m - C\varepsilon)\int_{\mathcal{DA}(\partial\Omega)}\nu(x_0)\cdot e(x_0)\,dH^{n-1}(x_0)\\
&\leq \int_{\mathcal{DA}(\partial\Omega)}\int^1_0|\nabla v(\gamma(t))| |\gamma'(t)|\,dt\,\nu(x_0)\cdot e(x_0)\,dH^{n-1}(x_0) \\
&\leq \int_{\mathcal{DA}(\partial\Omega)}\left(\int^1_0|\gamma'(t)|\,dt\right)^{\frac{1}{p'}}\left(\int^1_0|\nabla v(
\gamma(t))|^p|\gamma'(t)|\,dt\right)^{\frac{1}{p}}\nu(x_0)\cdot e(x_0)\,dH^{n-1}(x_0),\\
&= \int_{\mathcal{DA}(\partial\Omega)} L(\gamma_{x_0})^{\frac{1}{p'}}\left(\int^1_0|\nabla v(\gamma(t))|^p|\gamma'(t)|
\,dt\right)^{\frac{1}{p}}\nu(x_0)\cdot e(x_0)\,dH^{n-1}(x_0) \\
&\leq \left(\int_{\mathcal{DA}(\partial\Omega)}L(\gamma_{x_0})\nu(x_0)\cdot e(x_0)\,dH^{n-1}\right)^{\frac{1}{p'}}\left(\int_{\mathcal{DA}(\partial\Omega)}
\int^1_0|\nabla v(\gamma(t))|^p|\gamma'(t)|\nu \cdot e\,dt\,dH^{n-1}\right)^{\frac{1}{p}} \\
&\leq C|\Omega_1|^{\frac{1}{p'}}\left(\int_{\Omega_1}|\nabla v|^p\,dx\right)^{\frac{1}{p}}\\
&\leq C|\Omega|^{\frac{1}{p'}}\delta,
\end{split}
\end{equation*}
where the second and third inequalities are due to the application of the H\"{o}lder's inequality with $p' = \frac{p}{p-1}$, and the constant $C$ depends on $n$ and $p$. The second equality follows from the two representation formulas while the radius $r$ of $\Lambda$ is taken small
\begin{equation*}
\left|\Omega_1\right| = C(n)\int_{\mathcal{DA}(\partial\Omega)}\left(L(\gamma_{x_0}) + r - \frac{r^n}{(L(\gamma_{x_0})+r)^{n-1}}\right)\nu(x_0)\cdot e(x_0)\,dH^{n-1}(x_0)
\end{equation*}
and
\begin{equation*}
\int_{\Omega_1}\left|\nabla v(x)\right|^p\,dx = C(n)\int_{\mathcal{DA}(\partial\Omega)}\int^1_0\left|\nabla v(\gamma_{x_0}(t))\right|^p\,\left|\gamma'_{x_0}(t)\right|\nu(x_0) \cdot e(x_0)\,dt\,dH^{n-1}(x_0).
\end{equation*}

Since $C\varepsilon \ll \sigma_m$, we can take sufficiently small $\delta$ that is independent of $\varepsilon$ such that $$C(n,p)|\Omega|^{\frac{1}{p'}}\,\delta \leq \frac{1}{2}\sigma_m \int_{\mathcal{DA}(\partial\Omega)}\nu(x_0)\cdot e(x_0)\,dH^{n-1}(x_0).$$ This leads to a contradiction. So, $\Lambda$ must be empty.

Finally we prove that $\|v\|_{\mathfrak{B}} = \delta$ implies
\begin{equation}
I[v]= \int_{\Omega}\frac{1}{p}|\nabla v + \nabla u_0|^p - \frac{1}{p}|\nabla u_0|^p \geq a\ \ \text{for a certain $a > 0$.}
\end{equation}

If $p\geq 2$, then the elementary inequality (\ref{ele3}) implies that
\begin{alignat*}{1}
I[v] &= \int_{\Omega}\frac{1}{p}\left|\nabla v+ \nabla u_0\right|^p - \frac{1}{p}\left|\nabla u_0\right|^p \\
&\geq \int_{\Omega}\langle\left|\nabla u_0\right|^{p-2}\nabla u_0, \nabla v\rangle + C(p)\left|\nabla v\right|^p \\
&= C(p)\int_{\Omega}\left|\nabla v\right|^p = C(p)\delta^p > 0,
\end{alignat*}
while if $1 < p < 2$, then the elementary inequality (\ref{ele4}) implies
\begin{alignat*}{1}
I[v] &\geq p(p-1)\int_{\Omega}\left|\nabla v\right|^2\int^1_0\int^t_0\frac{1}{\left|\nabla u_0 + s\nabla v\right|^{2-p}}\,ds\,dt\,dx \\
&\geq p(p-1)\int_{\Omega}\left|\nabla v\right|^2\int^1_0\int^t_0\frac{1}{\left(\left|\nabla u_0\right| + s\left|\nabla v\right|\right)^{2-p}}\,ds\,dt\,dx. 
\end{alignat*}
If $\int_{\Omega}|\nabla u_0|^p = 0$, then $I[v] = \frac{1}{p}\delta^p > 0$. So in the following, we assume $\int_{\Omega}|\nabla u_0|^p > 0$.

Let $S = S_{\lambda} = \{x\in\Omega\colon |\nabla v| > \lambda\delta\}$, where the constant $\lambda = \lambda(p,|\Omega|)$ is to be taken. Then
\begin{alignat*}{1}
\delta^p &= \int_{\Omega}|\nabla v|^p = \int_{\{|\nabla v|\leq \lambda\delta\}}|\nabla v|^p + \int_S|\nabla v|^p \\
&\leq (\lambda\delta)^p|\Omega| + \int_S|\nabla v|^p
\end{alignat*}
and hence
\begin{equation*}
\int_S|\nabla v|^p \geq \delta^p\left(1 - \lambda^p|\Omega|\right) \geq \frac{1}{2}\delta^p,\ \ \text{if $\lambda$ satisfies\ }\frac{1}{4} < \lambda^p|\Omega| \leq \frac{1}{2}.
\end{equation*}
Meanwhile, for $1 < p < 2$, it holds that
\begin{alignat*}{1}
I[v] &\geq C(p)\int_{S}\left|\nabla v\right|^2\int^1_0\int^t_0\frac{1}{\left(\left|\nabla u_0\right| + s\left|\nabla v\right|\right)^{2-p}}\,ds\,dt\,dx \\
&=C(p)\left(\int_{S\cap\{|\nabla u_0|\leq |\nabla v|\}}\left|\nabla v\right|^2\int^1_0\int^t_0\frac{1}{\left(|\nabla u_0| + s|\nabla v|\right)^{2-p}}\,ds\,dt\,dx \right. \\
&\ \ \ \ + \left.\int_{S\cap\{|\nabla u_0| > |\nabla v|\}}\left|\nabla v\right|^2\int^1_0\int^t_0\frac{1}{\left(\left|\nabla u_0\right| + s\left|\nabla v\right|\right)^{2-p}}\,ds\,dt\,dx\right).
\end{alignat*}
The first integral on the right satisfies
\begin{alignat*}{1}
&\ \ \ \ \int_{S\cap\{|\nabla u_0|\leq |\nabla v|\}}\left|\nabla v\right|^2\int^1_0\int^t_0\frac{1}{\left(|\nabla u_0| + s|\nabla v|\right)^{2-p}}\,dsdtdx \\
&\geq \int_{S\cap\{|\nabla u_0|\leq |\nabla v|\}}\left|\nabla v\right|^p\int^1_0\int^t_0\frac{1}{\left(1 + s\right)^{2-p}}\,dsdtdx \\
&= C(p)\int_{S\cap\{|\nabla u_0|\leq |\nabla v|\}}\left|\nabla v\right|^p\,dx,
\end{alignat*}
while the second integral on the right satisfies
\begin{alignat*}{1}
&\ \ \ \ \int_{S\cap\{|\nabla u_0| > |\nabla v|\}}\left|\nabla v\right|^2\int^1_0\int^t_0\frac{1}{\left(\left|\nabla u_0\right| + s\left|\nabla v\right|\right)^{2-p}}\,dsdtdx \\
&\geq \int_{S\cap\{|\nabla u_0| > |\nabla v|\}}\frac{\left|\nabla v\right|^2}{|\nabla u_0|^{2-p}}\int^1_0\int^t_0\frac{ds\,dt}{(1+s)^{2-p}}\,dx \\
&= C(p) \int_{S\cap\{|\nabla u_0| > |\nabla v|\}}\frac{\left|\nabla v\right|^2}{|\nabla u_0|^{2-p}}\,dx.
\end{alignat*}
The H\"{o}lder's inequality applied with exponents $\frac{2}{p}$ and $\frac{2}{2-p}$ implies that
\begin{equation*}
\int_{S\cap\{|\nabla u_0| > |\nabla v|\}}\left|\nabla v\right|^p \leq \left(\int_{S\cap\{|\nabla u_0| > |\nabla v|\}}\frac{|\nabla v|^2}{|\nabla u_0|^{2-p}}\right)^{\frac{p}{2}}\left(\int_{S\cap\{|\nabla u_0| > |\nabla v|\}}|\nabla u_0|^p\right)^{\frac{2-p}{2}},
\end{equation*}
or equivalently
\begin{alignat*}{1}
\int_{S\cap\{|\nabla u_0| > |\nabla v|\}}\frac{|\nabla v|^2}{|\nabla u_0|^{2-p}} &\geq \frac{\left(\int_{S\cap\{|\nabla u_0| > |\nabla v|\}}\left|\nabla v\right|^p\right)^{\frac{2}{p}}}{\left(\int_{S\cap\{|\nabla u_0| > |\nabla v|\}}|\nabla u_0|^p\right)^{\frac{2-p}{p}}} \\
&\geq \frac{\left(\int_{S\cap\{|\nabla u_0| > |\nabla v|\}}\left|\nabla v\right|^p\right)^{\frac{2}{p}}}{\left(\int_{\Omega}|\nabla u_0|^p\right)^{\frac{2-p}{p}}}.
\end{alignat*}
Consequently,
\begin{alignat*}{1}
I[v] &\geq C(p)\int_{S\cap\{|\nabla u_0|\leq |\nabla v|\}}|\nabla v|^p + C(p)\frac{\left(\int_{S\cap\{|\nabla u_0| > |\nabla v|\}}\left|\nabla v\right|^p\right)^{\frac{2}{p}}}{\left(\int_{\Omega}|\nabla u_0|^p\right)^{\frac{2-p}{p}}} \\
&\geq C(p)\left(\int_{S\cap\{|\nabla u_0|\leq |\nabla v|\}}|\nabla v|^p\right)^{\frac{2}{p}} + C(p)\frac{\left(\int_{S\cap\{|\nabla u_0| > |\nabla v|\}}\left|\nabla v\right|^p\right)^{\frac{2}{p}}}{\left(\int_{\Omega}|\nabla u_0|^p\right)^{\frac{2-p}{p}}},\ \ \text{as $\delta$ is small} \\
&\geq C(p)A(u_0)\left(\left(\int_{S\cap\{|\nabla u_0|\leq |\nabla v|\}}|\nabla v|^p\right)^{\frac{2}{p}} + \left(\int_{S\cap\{|\nabla u_0|> |\nabla v|\}}|\nabla v|^p\right)^{\frac{2}{p}}\right) \\
&\geq C(p)A(u_0)\left(\int_{S}|\nabla v|^p\right)^{\frac{2}{p}} = C(p)A(u_0)\delta^2,
\end{alignat*}
where the last inequality is a consequence of the elementary inequality
\begin{equation*}
a^{\frac{2}{p}} + b^{\frac{2}{p}} \geq C(p)\left(a+b\right)^{\frac{2}{p}}\ \ \text{for\ }\ a, b\geq 0,
\end{equation*}
and the constant
\begin{equation*}
A(u_0) = \min\left\{1,\frac{1}{\left(\int_{\Omega}|\nabla u_0|^p\right)^{\frac{2-p}{p}}}\right\}.
\end{equation*}

So we have proved $I[v] \geq a > 0$ for some $a> 0$ whenever $v\in C^{\infty}_0(\Omega)$ satisfies $\|v\|_{\mathfrak{B}} = \delta$, for any $p\in (1, \infty)$.
\end{pf}

Let
\begin{equation*}
\mathcal{G} = \{\gamma\in C([0,1],H): \gamma(0) = 0\ \text{and\ }\gamma(1) = v_2\}
\end{equation*}
and
\begin{equation*}
c = \inf_{\gamma\in\mathcal{G}}\max_{0\leq t\leq 1}I[\gamma(t)].
\end{equation*}
The verified Palais-Smale condition and the preceding lemma allow us to apply the Mountain Pass Theorem as stated, for example, in \cite{J} to conclude that there is a $v_1\in \mathfrak{B}$ such that $I[v_1] = c$,
and $I'[v_1] = 0$ in $\mathfrak{B}^*$.
That is
\begin{equation*}
\int_{\Omega}|\nabla u|^{p-2}\nabla u\cdot\nabla\varphi + \left(u\theta(u) + \Theta_{\varepsilon}(u) + q(x)\mu_{\varepsilon}(u)\right)\varphi\, dx = 0
\end{equation*}
for any $\varphi\in \mathfrak{B} = W^{1,p}_0(\Omega)$, where $u_1 = v_1 + u_0$. So $u_1$ is a weak solution of the problem (\ref{euler})-(\ref{bdry}). This concludes the proof of Theorem \ref{thirdsolution}.

In essence, the Mountain Pass Theorem is a way to produce a saddle point solution. Therefore, in general, $u_1$ tends to be an unstable solution in contrast to the stable solutions $u_0$ and $u_2$.

\section{Radially Symmetric Solutions to the Unregularized Problem}

In this section we discuss a special case that beautifully illustrates the bifurcation phenomenon.  Specifically, we assume that $\Omega = B_R(0) \subset\mathbb{R}^n$, taking the outer radius $R = 1$ for convenience, and assume that the solution being discussed is radially symmetric. In addition we take the function $Q(x) = 1$. As such, the system \eqref{euler}-\eqref{bdry} is now given by the nonlinear ordinary differential equation

\begin{equation}\label{DE_general}
\left\{\begin{array}{ll} 
-(p-1)u_r^{p-2} u_{rr} - (n-1) \frac{u_r^{p-1}}{r} + \theta_\varepsilon (u) u + \Theta_\varepsilon (u) + \mu_\varepsilon(u) = 0 \quad & {\rm in}~ B_1(0)\\[1ex]
~~u = \sigma & {\rm on}~\partial B_1\\\end{array}\right.
\end{equation}
where we recall $\theta_{\varepsilon}(s) = \Theta'_{\varepsilon}(s)$, and $\mu_{\varepsilon}(s) = \lambda'_{\varepsilon}(s)$.  Given constant boundary data $\sigma$,  
we fix the regularization parameter $0< \varepsilon << \sigma$. The radius $r_0$ is then given by $u(r_0) = 0$ and the radius $r_1$ is given by $u(r_1) = \varepsilon$.  Substituting allows us to rewrite the system \eqref{DE_general} in the form 
\begin{equation}\label{DE_1}
\left\{\begin{array}{ll} 
-(p-1)u_r^{p-2} u_{rr} - (n-1) \frac{u_r^{p-1}}{r} = -\Delta_p u = 0 & {\rm in}~r_1 < |x| < 1 \\[1ex]
-(p-1)u_r^{p-2} u_{rr} - (n-1) \frac{u_r^{p-1}}{r} + \theta_\varepsilon (u) u + \Theta_\varepsilon (u) + \mu_\varepsilon(u) = 0 \quad &{\rm in}~ r_0 < |x| < r_1 \\[1ex]
-(p-1)u_r^{p-2} u_{rr} - (n-1) \frac{u_r^{p-1}}{r} = -\Delta_p u = - 1 & {\rm in}~ |x| < r_0\\[1ex]
~~u = \sigma & {\rm on}~|x| = 1
\end{array}\right.
\end{equation}
How many solutions exist to \eqref{DE_1}?  We expect to be able to explicitly show the bifurcation guaranteed by the main theorem, but even in the ODE setting this is not a trivial problem. 










As a further simplification, we consider the unregularized case \eqref{unreg}.  As such, $\varepsilon= 0$ and the free boundary at $|x| = r_0$ has a Bernoulli condition $(u_\nu^+)^p = (u_\nu^-)^p + \Lambda$, where $\Lambda=\lambda^p_2-\lambda^p_1$ from \eqref{fcts3}.
The system is now of the form
\begin{equation}\label{DE_2}
\left\{
\begin{array}{ll}
\Delta_p u = 0 & {\rm in}~ r_0 < |x| < 1\\[1ex]
\Delta_p u = 1 &  {\rm in}~|x| < r_0\\[1ex]
u = 0,\quad (u_\nu^+)^p = (u_\nu^-)^p + \Lambda\quad & {\rm on}~|x| = r_0\\[1ex]
u = \sigma & {\rm on}~|x| = 1
\end{array}\right.
\end{equation}
as is illustrated in Figure \ref{unreg_DE_domain}.

\begin{figure}[ht]
\centering

\begin{tikzpicture}[scale=3]

\def\R{1}
\def\rzero{0.35}

\draw (0,0) circle (\R);
\draw (0,0) circle (\rzero);

\fill (0,0) circle (0.02);

\draw[->, thick] (0,0) -- ({\R*cos(30)},{\R*sin(30)});
\node at (0.50,0.42) {$1$};

\draw[->, thick] (0,0) -- ({\rzero*cos(-30)},{\rzero*sin(-30)});
\node at (0.12,-0.18) {$r_0$};

\end{tikzpicture}
\caption{Domain for the system \eqref{DE_2}}
\label{unreg_DE_domain}

\end{figure}
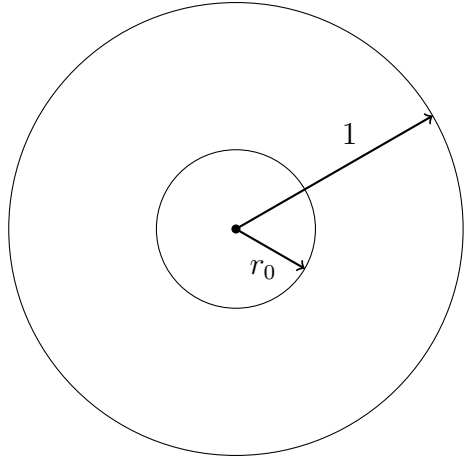

\subsection{Case $p=2, n>2$}
We first consider the case of the ordinary Laplacian, $p = 2$. Using $r = |x|$, and $u^+, u^-$ to denote the solution where $\{u > 0\}, \{u < 0\}$, respectively, the system becomes
\begin{equation}\label{DE_lap}
\left\{
\begin{array}{ll}
\Delta u = 0 & {\rm in}~  r_0 < r < 1\\[1ex]
\Delta u = 1 &  {\rm in}~ r < r_0\\[1ex]
u = 0,\quad (u_r^+)^2 = (u_r^-)^2 + \Lambda\quad & {\rm on}~r = r_0\\[1ex]
u = \sigma & {\rm on}~ r = 1
\end{array}\right.
\end{equation}
In the harmonic region $r_0 < r < 1$, the boundary conditions $u = 0, r = r_0$ and $u = \sigma, r = 1$ give
\begin{equation}\label{lap_0}
u^+(r) =\sigma\frac{r^{2-n} - r_0^{2-n}}{1-r_0^{2-n}} 
\end{equation}
In the region $u<r_0$, $\Delta u =1$ becomes
$$u_{rr} + \frac{n-1}{r}u_r = 1$$
so applying an integrating factor allows us to solve for $u_r^-$ with unknown constant $A$ as
$$u_r^-(r) = \frac{1}{n} r + \frac{A}{r^{n-1}}$$
Noting that $|u|<\infty$, integrating back and applying the boundary condition $u(r_0) = 0$ gives
\begin{equation}\label{lap_1}
u^-(r) = \frac{1}{2n}(r^2-r_0^2)
\end{equation}
We  use \eqref{lap_0} and \eqref{lap_1} in the jump condition at $r_0$ given by $(u_r^+)^2 = (u_r^-)^2 + \Lambda$ to derive the following relationship between $\sigma$ and $r_0$:
\begin{equation}\label{sigma_fct_p2}
\sigma = f(r_0) = \frac{1}{2-n}\left(r_0^{n-1}-r_0\right)\sqrt{\left( \frac{r_0^2}{n^2} + \Lambda\right)}
\end{equation}
The function $f(r_0)$ can be shown to be unimodal since $r_0^{n-1}-r_0$ is unimodal and the $\sqrt{\left( \frac{r_0^2}{n^2} + 1\right)}$ is monotonic and log-concave.  As a helpful visual, we graph the relationship for values of $n=3,~ n=4$ in Figure \ref{sigma_graph}. 
\begin{figure}[ht]
\centering
\includegraphics[width=0.6\textwidth]{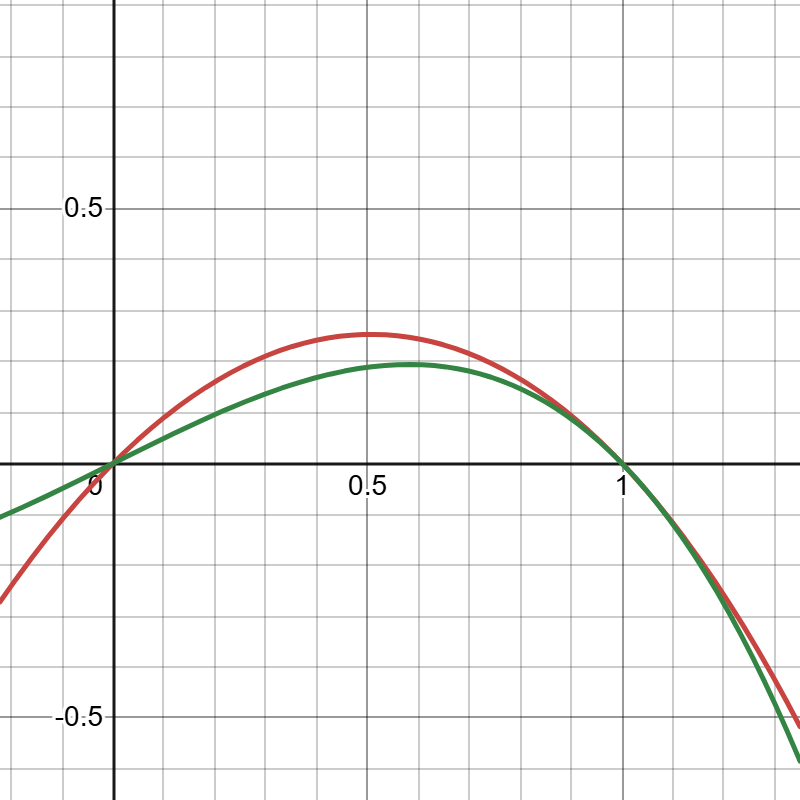}
\caption{$\sigma$ as a function of $r_0$ for $n=3$ (red) and $n=4$ (green)}
\label{sigma_graph}
\end{figure}
As such, for a given value of $\sigma$ we can see that the system admits zero, one or two solutions. 

\subsection{General case,  $p\ne n$}

We consider now the general case for any $p$ and $n$ with $p\ne n$. In this case our differential equation becomes 
\begin{equation}\label{DE_unreg_plap}
\left\{\begin{array}{ll} 
(p-1)u_r^{p-2} u_{rr} + (n-1)\frac{u_r^{p-1}}{r} = 0 & {\rm in}~r_0 < r < 1 \\[2ex]
(p-1)u_r^{p-2} u_{rr} + (n-1)\frac{u_r^{p-1}}{r} =  1 & {\rm in}~ r  < r_0\\[2ex]
u = 0,\quad (u_r^+)^p = (u_r^-)^p + \Lambda\quad & {\rm on}~r = r_0\\[1ex]
u = \sigma & {\rm on}~ r = 1
\end{array}\right.
\end{equation}
Multiplying by $u_r$, rearranging, applying integrating factor and the boundary conditions gives the following solution in the region $r_0<r<1$:
$$u^+(r) = \sigma\left(\frac{r^m-r_0^m}{1-r_0^m}\right)$$
so that 
\begin{equation}\label{urplus}
u^+_r(r) = \frac{\sigma}{(1-r_0^m)} mr^{m-1}
\end{equation}
where $m = \frac{p-n}{p-1}$.
Similarly, for $r<r_0$, we multiply by $r^{n-1}$, integrate, and apply the boundary conditions which (after noting that $|u_r(0)|<\infty$) gives
$$u^-(r) = \frac{p-1}{p}\left(\frac{1}{n}\right)^{\frac{1}{p-1}} \left( r^\frac{p}{p-1} - r_0^\frac{p}{p-1}\right)$$
so that
\begin{equation}\label{urminus}
u^-_r(r) = \left(\frac{r}{n}\right)^{\frac{1}{p-1}}
\end{equation}
Now, after applying the Bernoulli condition at $r=r_0$ and solving for $\sigma$ we have a function of the form
\begin{equation}\label{sigma_fct_p}
\sigma = f(r_0) = \frac{1}{m}\left( r_0^{1-m} - r_0\right)\left[ \left(\frac{r_0}{n}\right)^\frac{p}{p-1} + \Lambda\right]^\frac{1}{p}
\end{equation}
which is consistent with the analogous equation \eqref{sigma_fct_p2}.
We note that $f(0) = 0$ and $f(1) = 0$. We claim that this $f$, viewed as a function of the variable $r_0$, is a unimodal function with a unique maximum that is achieved at some value $r^*$ between zero and one. Thus, any value of the parameter $\sigma$ will result in either zero, one or two possible values of $r_0$ that solve the equation \eqref{sigma_fct_p2}.

This can be shown by decomposing $f$ in the following manner:
\begin{equation}\label{fgphi}
f(r_0) = \frac{1}{m}\left( r_0^{1-m} - r_0\right)\left[ \left(\frac{r_0}{n}\right)^\frac{p}{p-1} + 1\right]^\frac{1}{p} = g(r_0) \varphi(r_0)
\end{equation}
with
\begin{equation}\label{gphi_def}
g(r_0) = \frac{1}{m}\left( r_0^{1-m} - r_0\right);\qquad \varphi(r_0) = \left[ \left(\frac{r_0}{n}\right)^\frac{p}{p-1} + 1\right]^\frac{1}{p}
\end{equation}
We can see that 
\begin{equation}\label{f_unimodal}
f'(r_0) = \varphi(r_0)\left[g'(r_0) + g(r_0) \frac{\varphi'(r_0)}{\varphi(r_0)}\right]
\end{equation}
The function $\varphi(r_0)$ is positive, so the sign of $f'$ is determined by the expression in brackets.  Now, the function $g(r_0)\ge 0$ on $[0,1]$ and we can check that
$g(r_0)$ is a unimodal function: 
$$g'(r_0)   = \frac{1}{m}\left( (1-m)r_0^{-m} - 1\right)$$
For $p<n$, $m<0$ so $g'(0) = \frac{p-1}{n-p} >0$ and for $n<p$, $m>0$ so $g'(r_0)\rightarrow \infty$ as $r_0\rightarrow 0$.  For both cases, $g'(1) = -1$; and solving for where $g'(r_0^*) = 0$ gives a unique maximum for $g$ at $r_0^* = \left(\frac{p-1}{n-1}\right)^{-\frac{p-1}{p-n}}\in (0,1)$.
In contrast, the function $\varphi(r_0)$ is positive and monotone increasing on $[0,1]$.  In addition, $\varphi$ is log-convave.  That is to say, $(\frac{\varphi'}{\varphi})'<0$. By taking the derivative of the expression in brackets in \eqref{f_unimodal}, applying the facts above and the intermediate value theorem, we can show that $f(r_0)$ has exactly one maximum point in the interval $(0,1)$.

\subsection{General case, $p=n$}

With $p=n$ the system becomes
\begin{equation}\label{DE_unreg_peqn}
\left\{\begin{array}{ll} 
u_r^{n-2} u_{rr} + \frac{u_r^{n-1}}{r} = 0 & {\rm in}~r_0 < r < 1 \\[2ex]
(n-1)u_r^{n-2} u_{rr} + (n-1)\frac{u_r^{n-1}}{r} =  1 & {\rm in}~ r  < r_0\\[2ex]
u = 0,\quad (u_r^+)^n = (u_r^-)^n + \Lambda\quad & {\rm on}~r = r_0\\[1ex]
u = \sigma & {\rm on}~ r = 1
\end{array}\right.
\end{equation}
Solving  in the same way on the region $r_0<r<1$ gives
$$u^+(r) = u(r)=\sigma\left(1-\frac{\log r}{\log r_0}\right)$$
and on $r<r_0$ we have
$$u^-(r) =
\frac{n-1}{n^{\,n/(n-1)}}
\left(
r^{\,n/(n-1)}-r_0^{\,n/(n-1)}
\right)$$
Applying the Bernouilli condition at $r=r_0$ and solving for $\sigma$ gives the function
$$ \sigma =  f(r_0) = -r_0\log r_0
\left[
\left(\frac{r_0}{n}\right)^{\frac{n}{n-1}} 
+ \Lambda
\right]^{1/n}. $$
We can rewrite this as $f(r_0) = g(r_0)\varphi(r_0)$ as before.  Notice that $\varphi$ is identical to the previous case, though $g(r_0)$ is different.  It is easy to check that this $g(r_0)$ is also unimodal,  with unique maximum at $r_0^* = 1/e$.  Thus by the previous argument, this $f(r_0)$ is also unimodal and we see the expected bifurcation of solutions.

\end{document}